\documentclass[11pt,a4paper]{amsart}

\usepackage[utf8]{inputenc}
\usepackage{amsmath,amsthm,amssymb,mathtools}
\usepackage{graphicx}
\usepackage{xcolor}
\usepackage{aliascnt}
\usepackage{fullpage}
\usepackage{needspace}
\usepackage[noadjust]{cite}
\usepackage[hidelinks]{hyperref}
\usepackage[noabbrev,capitalize,nameinlink]{cleveref}
\usepackage[shortlabels]{enumitem}
\usepackage{bm}

\allowdisplaybreaks[1]
\numberwithin{equation}{section}
\crefname{equation}{}{}

\theoremstyle{plain}
\newtheorem{theorem}{Theorem}[section]
\newtheorem*{theorem*}{Theorem}

\newaliascnt{proposition}{theorem}

\aliascntresetthe{proposition}
\crefname{proposition}{Proposition}{Propositions}

\newaliascnt{lemma}{theorem}
\newtheorem{lemma}[lemma]{Lemma}
\aliascntresetthe{lemma}
\crefname{lemma}{Lemma}{Lemmas}

\newaliascnt{corollary}{theorem}
\newtheorem{corollary}[corollary]{Corollary}
\aliascntresetthe{corollary}
\crefname{corollary}{Corollary}{Corollaries}

\theoremstyle{remark}
\newaliascnt{remark}{theorem}

\aliascntresetthe{remark}
\crefname{remark}{Remark}{Remarks}

\newcommand{\mc}[1]{\mathcal{#1}}

\newcommand{\RR}{\mathbb{R}}

\newcommand{\E}{\mathbb{E}}
\newcommand{\PP}{\mathbb{P}}
\newcommand{\ZZ}{\mathbb{Z}}

\newcommand{\SL}{\mathrm{SL}}
\newcommand{\LS}{\mathrm{LS}}
\newcommand{\KL}{\mathrm{KL}}

\renewcommand{\le}{\leqslant}
\renewcommand{\ge}{\geqslant}

\newcommand{\inner}[2]{\langle #1,#2 \rangle}
\newcommand{\norm}[1]{\left \| #1 \right \|}
\newcommand{\abs}[1]{\left \lvert #1 \right \rvert}

\DeclareMathOperator{\Cov}{Cov}
\DeclareMathOperator{\Id}{Id}
\DeclareMathOperator{\Ent}{Ent}
\DeclareMathOperator{\im}{im}
\DeclareMathOperator{\tr}{tr}
\DeclareMathOperator{\rank}{rank}
\DeclareMathOperator{\spn}{span}
\DeclareMathOperator{\Var}{Var}
\DeclareMathOperator{\Vol}{Vol}
\DeclareMathOperator{\Flt}{Flt}
\DeclareMathOperator{\intr}{int}
\DeclareMathOperator{\Unif}{Unif}

\title{Strongly Log-Concave Measures and Lattices}
\hypersetup{
    pdftitle={Strongly Log-Concave Measures and Lattices},
    pdfauthor={Brayden Letwin, Victor Reis, Thomas Rothvoss}
}

\author[Letwin]{Brayden Letwin}
\address{Department of Mathematics, University of Washington, Seattle, WA 98195}
\email{letwin@uw.edu}

\author[Reis]{Victor Reis}
\address{Microsoft Research}
\email{victorol@microsoft.com}

\author[Rothvoss]{Thomas Rothvoss}
\address{Department of Mathematics, University of Washington, Seattle, WA 98195}
\email{rothvoss@uw.edu}
\thanks{Rothvoss was supported by NSF grant 2318620 \emph{AF: SMALL:
The Geometry of Integer Programming and Lattices}.}

\begin{document}

\begin{abstract}
Let $\Lambda \subseteq \RR^n$ be a full rank lattice with the property that $\det \Gamma \ge 1$ for all sublattices $\Gamma \subseteq \Lambda^*$. We prove that there exists a symmetric random vector $X \in \RR^n$ with a smooth, strictly positive density $f : \RR^n \to (0, \infty)$ satisfying that
\begin{enumerate}
    \item $f$ is $\alpha$-strongly log-concave, where $\alpha = 2\pi/e$.
    \item $X + \Lambda$ is uniform on $\RR^n/\Lambda$.
\end{enumerate}
Moreover, the existence of such a density characterizes this property up to a universal dilation. There are a few consequences of the above. We list some. First, for such lattices we can bound the covering radius 
\[\mu(\Lambda, K) =\max_{x\in\RR^n}\min_{\lambda\in\Lambda}\norm{x-\lambda}_K \le \sqrt{\frac{2e}{\pi}} \cdot \E\, \norm{G}_K\] of a convex body $K \subseteq \RR^n$ containing $0$ in its interior, where $G \sim N(0,\Id)$. This implies an $O(n \sqrt{\log n})$ bound for the flatness constant and an $O(\log^{3/2} n)$ bound for the subspace flatness constant. We also compare the entire ordered spectrum of the associated flat torus on $\Lambda^*$ with that of a Gaussian, yielding an alternative proof of the reverse Minkowski theorem.
\end{abstract}
\maketitle
\section{Introduction}\label{sec:1}
A \textit{lattice} is a discrete additive subgroup of $\RR^n$. Lattices play a central role in cryptography and theoretical computer science, as well as in number theory and the geometry of numbers. Every lattice admits linearly independent vectors $b_1, \ldots, b_k \in \RR^n$, called a \textit{basis}, satisfying that $\Lambda = \spn_{\ZZ}\{b_1, \ldots, b_k\}$. The number $k$ is the \textit{rank} of the lattice. When $k = n$, we say that the lattice has \textit{full rank} and its \textit{determinant} is given by
\[
    \det(\Lambda) = \abs{\det(b_1, \ldots, b_n)}.
\]
For a lower-rank lattice, the determinant is computed in its linear span. The determinant is independent of the chosen basis, and one has that $1/\det \Lambda$ captures the "average density" of $\Lambda$. Precisely, for a full-rank lattice,
\[
    \lim_{R \to \infty} \frac{\abs{\Lambda \cap R B_2^n}}{\Vol_n(R B_2^n)} = \frac{1}{\det \Lambda}.
\]
We recommend Gruber and Lekkerkerker \cite{GeometryOfNumbers-GruberLekkerkerker1987} for an overview of lattices. A \textit{sublattice} $\Gamma \subseteq \Lambda$ of $\Lambda$ is an additive subgroup of $\Lambda$, which is itself a lattice because $\Gamma$ is automatically discrete. Any sublattice $\Gamma \subseteq \Lambda$ that is of the form $\Lambda \cap E$ where $E$ is a subspace of $\RR^n$ is called \textit{primitive}. For such a primitive sublattice, the set $\Lambda/\Gamma = \Pi_{ (\spn_{\RR}\Gamma)^{\perp}} \Lambda$ is a lattice, called a \textit{quotient lattice}. Here $E^{\perp}$ denotes the orthogonal complement of a subspace $E \subseteq \RR^n$, and $\Pi_E$ the orthogonal projection onto $E$. For a full-rank lattice, its \textit{dual lattice} is $\Lambda^*=\{x\in\RR^n:\inner{x}{\lambda}\in\ZZ\text{ for every }\lambda\in\Lambda\}$. 

We call $\Lambda$ a \textit{stable lattice} if $\det \Lambda = 1$ and $\det \Gamma \ge 1$ for every sublattice $\Gamma \subseteq \Lambda$. In other words, a stable lattice $\Lambda$ is one having the property that no sublattice $\Gamma \subseteq \Lambda$ is more dense than itself. The work of Harder and Narasimhan\footnote{The terminology in such works has often used the term \textit{semistability} instead of stability, we adopt the latter term.} \cite{Harder1974}, Stuhler \cite{Stuhler1976}, and Grayson \cite{Grayson1984} shows that for any lattice $\Lambda \subseteq \RR^n$, there exists a filtration $\{0\}=\Lambda_0 \subsetneq \Lambda_1 \subsetneq \ldots \subsetneq \Lambda_k = \Lambda$ by primitive sublattices, called the \textit{canonical filtration} or \textit{stable filtration} satisfying two properties
\begin{enumerate}
    \item The \textit{normalized determinants} of $\Lambda_i/\Lambda_{i-1}$, denoted by $r_i = \det(\Lambda_i/\Lambda_{i-1})^{1/\rank(\Lambda_i/\Lambda_{i-1})}$, are \textit{strictly} increasing: $r_1 < \ldots < r_k$.
    \item Each scaled quotient lattice $\frac{1}{r_i}(\Lambda_i/\Lambda_{i-1})$ is stable.
\end{enumerate}
The stable filtration of $\Lambda$ is unique, and often many problems about arbitrary lattices can be reduced to the stable case by using the stable filtration. Stable lattices are also typical among determinant-one lattices. To state this precisely, let
\[
    \mc{L}_n = \{\Lambda \subseteq \RR^n : \Lambda \text{ is a full-rank lattice and } \det \Lambda = 1 \}
\]
be the set of all determinant one lattices in $\RR^n$. This space can be identified with $\SL_n(\RR)/\SL_n(\ZZ)$  and carries a unique $\SL_n(\RR)$-invariant probability measure $\mu_n$, called the Haar--Siegel measure; see \cite{OrdentlichRegevWeiss2022,Siegel1945}. Shapira and Weiss \cite{ShapiraWeiss2014} proved that
\[
    \PP_{\Lambda \sim \mu_n}(\Lambda \text{ is stable}) \to 1, \quad \text{as } n \to \infty.
\]
In fact, they showed that the rate at which the above probability converges is superexponential, with the probability that the above event fails at most $\exp(-\Omega(n \log n))$. For a full-rank lattice $\Lambda \subseteq \RR^n$, we will let $\lambda_1(\Lambda) = \min_{0 \neq \lambda \in \Lambda} \norm{\lambda}_2$ be the length of the shortest non-zero lattice point in $\Lambda$. Minkowski's first theorem \cite{GeometryOfNumbers-GruberLekkerkerker1987} states that
\[
\lambda_1(\Lambda) \le \sqrt{n} \cdot (\det \Lambda)^{1/n},
\]
i.e., if the determinant of $\Lambda$ is small (which is a global property about the "density" of $\Lambda$) then $\Lambda$ contains a short non-zero lattice point (which is a local property about the "density" of $\Lambda$). Another "global-to-local" result, generalizing the above, is that of Blichfeldt and van der Corput \cite{GeometryOfNumbers-GruberLekkerkerker1987} which states that
\[
    \abs{\Lambda \cap r B_2^n} \ge \frac{1}{2^n} \cdot \frac{\Vol_n(rB_2^n)}{\det \Lambda}.
\]
One may ask whether in general one can obtain a "local-to-global" result of the form, \textit{if there are many short lattice points, then does this imply that the lattice has small determinant?} It is easy to see that this question generally cannot be true, but a fundamental result, the \textit{reverse Minkowski theorem}, conjectured by Regev and Dadush~\cite{DadushRegev2016} and proven by Regev and Stephens--Davidowitz~\cite{RegevStephensDavidowitz2024} shows that there is a positive answer if one only cares about finding a sublattice with small determinant. For $s > 0$ and $x \in \RR^n$ let $\rho_s(x) = e^{-\pi \norm{x}_2^2/s^2}$ be the \textit{Gaussian mass} of $x$. In the estimates below we assume $n\ge2$, the one-dimensional case being (usually) immediate. The reverse Minkowski theorem is stated as follows: there exists a (sufficiently large) universal constant $C > 0$ so that for any lattice $\Lambda \subseteq \RR^n$ with $\det\Gamma\ge1$ for every nonzero sublattice $\Gamma\subseteq\Lambda$ and for all $0 < s \le 1/(C\log n)$ one has that
\begin{equation} \label{eq:1.1}
    \sum_{\lambda \in \Lambda} \rho_s(\lambda) \le \frac{3}{2},
\end{equation}
or in other words, if the above sum is "large" then some non-zero sublattice $\{0\} \neq \Gamma \subseteq \Lambda$ has $\det \Gamma < 1$.
\begin{figure}[htbp]
    \centering
    \includegraphics[width=0.4\linewidth]{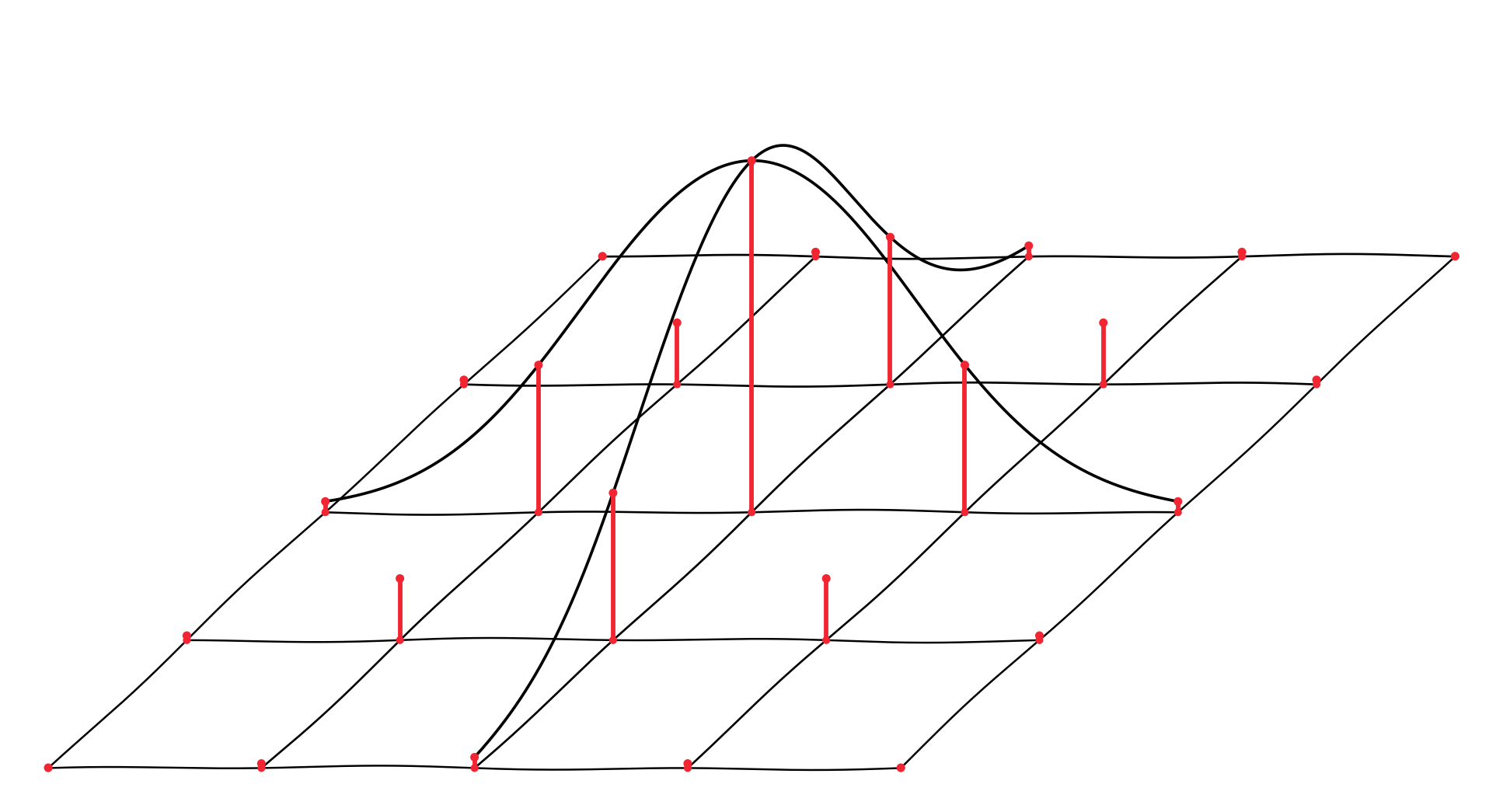}
    \caption{Gaussian masses on a lattice $\Lambda \subseteq \RR^2$.
    Each red segment extends from a lattice point $\lambda \in \Lambda$ to $\rho_s(\lambda)=e^{-\pi \norm{\lambda}_2^2/s^2}$.}
\end{figure}

The reverse Minkowski theorem is thus stated for a broader class than stable lattices: we say that a lattice $\Lambda \subseteq \RR^n$ is \textit{weakly stable}\footnote{The property of being stable is an established term in literature, but weak stability is not as far as we are aware.} if $\det \Gamma \ge 1$ for all sublattices $\Gamma \subseteq \Lambda$.
Of course every stable lattice is weakly stable. It is known that the sum in \cref{eq:1.1} can only be of constant order for $0 < s \le O(1/\sqrt{\log n})$, as can be shown by considering the case when $\Lambda = \ZZ^n$, so this is the sharp dependence on $n$ for $s$.

We should explain that the published work \cite{RegevStephensDavidowitz2024} uses a weaker $s = \Theta(1/\log n)$ parameter arising from the  classical $MM^*$-estimate due to Figiel and Tomczak-Jaegermann~\cite{FigielTomczakJaegermann1979}, Lewis~\cite{LewisEllipsoids1979} and Pisier~\cite{PisierKConvexity1982}.
Recent years have brought many breakthroughs in convex geometry by Chen~\cite{Chen2021}, Guan~\cite{guan2024notebourgainsslicingproblem},
Klartag and Lehec~\cite{KlartagLehecProofOfSlicingConjectureGAFA2025} and others, leading among other insights to a resolution of Bourgain's Slicing Conjecture. A result of this progress is the following
which can be found as a direct result of \cite{Bizeul2026,PaourisPathak2026}:
\begin{center}
  \textit{For any convex body $K \subseteq \RR^n$ with $\Vol_n(K)=1$, there is a volume-preserving affine map $T : \RR^n \to \RR^n$ so that $\gamma_n( C\sqrt{\log n} \cdot T(K)) \ge \frac{1}{2}$.}
\end{center}
Here $\gamma_n$ is the standard Gaussian measure and $C>0$ is a universal constant. Plugging this into the framework of \cite{RegevStephensDavidowitz2024} immediately gives \cref{eq:1.1} for $s=\Theta(1/\sqrt{\log n})$. Our main result, \cref{thm:1.1}, constructs a strongly log-concave probability measure on $\RR^n$ whose image modulo the lattice is uniform. We call this the \textit{log-concave lattice lift}. This lift allows Gaussian comparison principles to be applied to lattice problems. For example, Harg\'e's theorem~\cite{Harge2004} controls expectations of gauges and gives us immediate covering-radius bounds. Also, immediately the Bakry--\'Emery criterion~\cite{BakryGentilLedoux2014} gives functional inequalities for the uniform measure on $\RR^n/\Lambda$. Finally, Milman's theorem~\cite{Milman2018}, gives lower bounds for every eigenvalue on the torus. Summing these eigenvalues yields the reverse Minkowski theorem. These applications use different consequences of the log-concave lattice lift, we hope that many more can be found in future works.

\subsection{Our results}
By a slight abuse of notation, we write $\rho_s(S)=\sum_{x\in S}\rho_s(x)$ for the Gaussian mass of a discrete set $S$. For $\alpha > 0$, we say that a smooth function $f : \RR^n \to (0, \infty)$ is \textit{$\alpha$-strongly log-concave} if
\[
    \nabla^2(- \log f)(x) \succcurlyeq \alpha \cdot \Id, \quad \text{for every }x \in \RR^n.
\]
We now state our main result, \cref{thm:1.1}.
\begin{theorem}[Existence of the log-concave lattice lift]\label{thm:1.1}\label{thm:MainResult}
Let $\Lambda \subseteq \RR^n$ be a full-rank lattice such that $\Lambda^*$ is weakly stable. Then there exists a symmetric random vector $X \in \RR^n$ with a smooth, strictly positive density $f : \RR^n \to (0, \infty)$ satisfying that
\begin{enumerate}
    \item $f$ is $\alpha$-strongly log-concave, where $\alpha = 2\pi/e$.
    \item $X + \Lambda$ is uniform on $\RR^n/\Lambda$, i.e., $\sum_{\lambda \in \Lambda} f(x+\lambda) = \frac{1}{\det \Lambda}$ for all $x \in \RR^n$.
\end{enumerate}
As a converse, for any lattice $\Lambda$ for which such a density $f$ exists,  $ \sqrt{\frac{2\pi}{\alpha}} \cdot \Lambda^*$ is in fact weakly stable.
\end{theorem}
Thus \cref{thm:1.1} characterizes lattices with weakly stable duals, up to a universal dilation, by the existence of such densities $f$. We now demonstrate some implications of \cref{thm:1.1}. The \textit{covering radius} of a full-rank lattice $\Lambda \subseteq \RR^n$ with respect to a convex body $K \subseteq \RR^n$ is defined through
\[
\mu(\Lambda, K) = \min\{r \ge 0 : \Lambda + rK = \RR^n\}
\]
where $\norm{x}_K = \inf\{r > 0 : x \in rK\}$ is the Minkowski functional of $K$.
\begin{theorem}\label{thm:CoveringRadius}
    Let $\Lambda \subseteq \RR^n$ be a full-rank lattice so that $\Lambda^*$ is weakly stable and let $K \subseteq \RR^n$ be a convex body with $0\in\intr(K)$. Then
    \[
    \mu(\Lambda, K) \le \sqrt{\frac{2e}{\pi}} \cdot \E\, \norm{G}_K,
    \]
    where $G \sim N(0, \Id)$ is a standard Gaussian random vector in $\RR^n$.
\end{theorem}

As an application of this result, we obtain "flatness bounds". For a full rank lattice $\Lambda \subseteq \RR^n$, a convex body $K \subseteq \RR^n$ and a nonzero subspace $W \subseteq \RR^n$ for which $\Pi_W\Lambda$ is a lattice, we define the ratio
\[
R_W(\Lambda,K) = \left(\frac{\det(\Pi_W\Lambda)}
    {\Vol_{\dim W}(\Pi_WK)}\right)^{1/\dim W}.
\]
Kannan and Lov\'asz~\cite{KannanLovasz1988} showed that
$R_W(\Lambda,K)$ gives a lower bound on the covering radius. That is, $\mu(\Lambda,K) \ge R_W(\Lambda,K)$ for any such $W$.
We define the \textit{subspace flatness constant} $C_{\KL}(n)$ as the smallest number such that, for every convex body $K\subseteq\RR^n$ and full-rank lattice $\Lambda$,
\[
    \mu(\Lambda,K)\le C_{\KL}(n) \cdot \sup_W R_W(\Lambda,K)
\]
where the supremum is over
nonzero subspaces $W$ for which $\Pi_W\Lambda$ is a lattice.
Kannan and Lov\'asz~\cite{KannanLovasz1988} showed that $C_{\KL}(n) \le  n$ which was improved to  $C_{\KL}(n) \le  O(\log^3 n)$ by \cite{ReisRothvossFOCS23} (which would reduce to $C_{\KL}(n) \le  O(\log^{5/2} n)$ by applying the $O(\sqrt{\log n})$ bound for Reverse Minkowski).
Making use of Theorem~\ref{thm:CoveringRadius} instead, the bound improves further: 
\begin{corollary}\label{cor:SubspaceFlatness} One has that
$C_{\KL}(n)=O(\log^{3/2}n)$.
\end{corollary}

A classic quantity studied in the geometry of numbers is the 
\textit{flatness constant} $\Flt(n)$ in dimension $n$ (see e.g.~\cite{BanaszczykLitvakPajorSzarek1999}) which is defined as
\[
\begin{aligned}
  \Flt(n)&=\sup_{\Lambda,K}w_\Lambda(K),\\
  w_\Lambda(K)&=\min_{\lambda^*\in\Lambda^*\setminus\{0\}}
      \max_{x,y\in K}\inner{\lambda^*}{x-y}.
\end{aligned}
\]
Here $w_\Lambda(K)$ is the \textit{lattice width} of $K$ with respect to
$\Lambda$, and the supremum is over convex bodies $K\subseteq\RR^n$ and full-rank lattices $\Lambda$ such that $\intr(K)\cap\Lambda=\varnothing$. Equivalently, for any pair $K,\Lambda$ with $\Lambda \cap \intr(K) = \varnothing$ there exists a direction $\lambda^* \in \Lambda^* \setminus \{ 0\}$ in which $K$ is ``thin'', i.e. $\max_{x,y \in K} \inner{\lambda^*}{x-y} \le \Flt(n)$. We note that $\Flt(n)$ corresponds to the subspace flatness constant if subspaces were restricted to $\dim W=1$.
The previously best known bound is $\Flt(n) \le O(n \log n)$  by Bizeul~\cite{Bizeul2026} which built off that of Banaszczyk, Litvak, Pajor and Szarek~\cite{BanaszczykLitvakPajorSzarek1999}; the earlier work of \cite{ReisRothvossFOCS23,ReisRothvoss2026} gave $\Flt(n) \le O(n \log^2 n)$.
Here we prove:
\begin{corollary}\label{cor:Flatness} One has that
$\Flt(n)=O(n \sqrt{\log n})$.
\end{corollary}
We note that the best known lower bound is $\Flt(n) \ge (2-o(1))n$ due to Mayrhofer, Schade and Weltge~\cite{LatticeFreeSimplex-Arxiv.2111.08483}. Recall that a centered random vector $X \in \RR^n$ is called \textit{$s$-subgaussian} if
\[
    \E\, e^{\inner{\xi}{X}} \le e^{s^2 \norm{\xi}_2^2/2}, \quad \text{for every }\xi \in \RR^n.
  \]
For a probability measure $\mu$, the \textit{logarithmic Sobolev constant}
$C_{\LS}(\mu)$ is the least $C\ge0$ such that
\[
    \Ent_{\mu}(g^2) \le 2C \cdot \int \norm{\nabla g}_2^2\,d\mu
\]
for all smooth functions $g$ for which the expressions are finite;
we set $C_{\LS}(\mu)=+\infty$ if no finite $C$ exists.
By \cite[Corollary~5.7.2 and Proposition~5.4.1]{BakryGentilLedoux2014},
the law $\mu$ of an $\alpha$-strongly log-concave random vector satisfies
$C_{\LS}(\mu)\le1/\alpha$, and the vector is
$\alpha^{-1/2}$-subgaussian about its mean.
In particular, the random vector in \cref{thm:1.1} is $\sqrt{e/(2\pi)}$-subgaussian. We thus obtain immediately

\begin{corollary}\label{cor:MainResultII}
    Let $\Lambda \subseteq \RR^n$ be a full-rank lattice such that
    $\Lambda^*$ is weakly stable. Then the random vector $X$ from \cref{thm:1.1} is $\sqrt{\frac{e}{2\pi}}$-subgaussian, and moreover the uniform measure $\nu$ on $\RR^n/\Lambda$ satisfies the following log-Sobolev inequality:
    \[
    \Ent_{\nu}(g^2) \le \frac{e}{\pi} \cdot \int_{\RR^n/\Lambda} \norm{\nabla g}_2^2\,d\nu, \quad \text{for every }g \in C^{\infty}(\RR^n/\Lambda).
\]
\end{corollary}
\begin{proof}
    The subgaussian conclusion follows from the preceding bound with
    $\alpha=2\pi/e$. Let $\mu$ be the law of $X$, let
    $\pi(x)=x+\Lambda$, and suppose that
    $g\in C^\infty(\RR^n/\Lambda)$. We have
    \[
    \begin{aligned}
        \Ent_\nu(g^2)
        &=\Ent_\mu((g \circ \pi)^2)\\
        &\le\frac{2}{\alpha}
            \int_{\RR^n}\norm{\nabla(g \circ \pi)}_2^2\,d\mu\\
        &=\frac{e}{\pi}
            \int_{\RR^n/\Lambda}\norm{\nabla g}_2^2\,d\nu.
    \end{aligned}
    \]
    This proves the claim.
\end{proof}
Related and recent connections between the geometry of lattice fundamental
domains and functional inequalities on $\RR^n/\Lambda$ were studied by
Milman~\cite{Milman2026} who obtained spectral-gap bounds in terms of the covariance of a fundamental domain. \cref{thm:CoveringRadius} has a related formulation in terms of subgaussian distributions, connected to
the discrepancy theory work of Dadush, Garg, Lovett and Nikolov~\cite{DadushGargLovettNikolov2019}.
\begin{theorem}\label{thm:SubgaussianDistrForStableLattice}
    Let $\Lambda \subseteq \RR^n$ be a full-rank lattice so that $\Lambda^*$ is weakly stable and let $x \in \RR^n$. Then there exists a random vector $Y \in \Lambda$ so that $\E\, Y = x$, and $Y - x$ is $\sqrt{2e/\pi}$-subgaussian.
\end{theorem}

As a consequence of \cref{thm:1.1}, we obtain an alternative proof of the reverse Minkowski theorem.
\begin{theorem} \label{thm:1.2}\label{thm:ReverseMinkowskiWithTightConstants}\label{thm:ReverseMinkowski}
    Let $\Lambda \subseteq \RR^n$ be a weakly stable lattice. Then for all $0 < s \le \left(1/\sqrt{2e}+ o(1)\right)\cdot\frac{1}{\sqrt{\log n}}$ one has that
    \[
    \sum_{\lambda \in \Lambda} \rho_s(\lambda) \le \frac{3}{2}.
    \]
\end{theorem}
The dependence $1/\sqrt{\log n}$ is sharp, as follows by considering $\Lambda=\ZZ^n$.

\subsection{AI Disclosure}\label{subsec:AIdisclosure}
The authors used ChatGPT 5.6 Pro and ChatGPT 6 Pro in the development of this work.

\section{Preliminaries} \label{sec:Preliminaries}

\subsection{General notation}\label{subsec:1.2}
Before discussing the proof, let us fix the notation we will use.
All logarithms are natural unless a base is displayed.
We use $\subseteq$ for inclusion and $\subsetneq$ for proper inclusion.
We write $\inner{x}{y}$ for the Euclidean inner product on $\RR^n$,
$\norm{x}_2=\sqrt{\inner{x}{x}}$ for the Euclidean norm, and
$B_2^n=\{x\in\RR^n:\norm{x}_2\le1\}$ for the Euclidean unit ball.
The notation $\spn_{\RR}S$ denotes the real linear span of $S$, while
$\spn_{\ZZ}S$ denotes its integer span.

For a linear subspace $E\subseteq\RR^n$, we write $E^\perp$ for its
orthogonal complement and $\Pi_E$ for orthogonal projection onto $E$.
The identity on $\RR^n$ is denoted by $\Id$, and the identity on $E$ by
$\Id_E$. When regarded as an operator on $\RR^n$, an operator on $E$
is extended by zero on $E^\perp$; in particular, the extension of
$\Id_E$ is $\Pi_E$.
For a matrix $M$, the expressions $M^{\mathsf T}$, $\tr(M)$, and
$\norm{M}_{\mathrm{op}}$ denote its transpose, trace, and Euclidean
operator norm, respectively. For symmetric matrices $A$ and $A'$, the
\emph{Loewner (positive semidefinite) order} is defined by
\[
    A\preccurlyeq A'
    \quad\Longleftrightarrow\quad
    \inner{v}{Av}\le\inner{v}{A'v},
    \quad\text{for every }v\in\RR^n.
\]
For $A\succcurlyeq0$, the notation $A^{1/2}$ denotes its unique symmetric
positive semidefinite square root. Its \emph{Moore--Penrose pseudoinverse}
$A^+$ is the symmetric matrix satisfying
$\im(A^+)=\im(A)$ and $AA^+=A^+A=\Pi_{\im(A)}$.
We write $\Vol_k$ for $k$-dimensional Lebesgue measure and $\intr(K)$
for the interior of a set $K$ in its ambient space.

We denote probability and expectation by $\PP$ and $\E$, and covariance
by $\Cov$. The notation $X\sim\mu$ means that $X$ has law $\mu$;
for a law with finite second moments, $\Cov(\mu)=\Cov(X)$.
We write $U\sim\Unif(F)$ when $U$ is uniform on a measurable set $F$
of finite, positive volume.
The standard Gaussian measure is $\gamma_n=N(0,\Id)$.
For $\alpha>0$, we write $\gamma_{n,\alpha}=N(0,\alpha^{-1}\Id)$,
so $\gamma_n=\gamma_{n,1}$.

For $A\succ0$, we write $\rho_A(x)=e^{-\pi\inner{x}{A^{-1}x}}$
for the unnormalized Gaussian weight with matrix parameter $A$.
The scalar-width convention is $\rho_s=\rho_{s^2\Id}$ for $s>0$;
we abbreviate $\rho_1=\rho_{\Id}$ by $\rho$.
The normalized continuous Gaussian proportional to $\rho_A$ has
covariance $A/(2\pi)$. For a discrete set $S$, we write
$\rho_A(S)=\sum_{y\in S}\rho_A(y)$, with the same convention for
$\rho_s(S)$ and $\rho(S)$. The corresponding discrete probability
law is denoted by $D_{x+\Lambda,A}$ and is defined in Section~\ref{sec:3}.
For a smooth function $f:\RR^n\to\RR$, the symbols $\nabla f$ and
$\nabla^2 f$ denote its gradient and Hessian.

\subsection{Concentration}
The following properties are classical and well-known. We state them for completeness.
\begin{lemma}\label{lem:ConcentrationForAlphaLogConcave}
Let $\alpha>0$ and let $X$ have an $\alpha$-strongly log-concave probability density $f:\RR^n\to (0, \infty)$, with law $\mu=f\,dx$.
The following hold, with the additional assumption $\E\,X=0$ in (iii)--(v).
\begin{enumerate}
\item[(i)] \textnormal{(Bakry--\'Emery criterion~\cite[Corollary~5.7.2]{BakryGentilLedoux2014})}
For every smooth $g:\RR^n\to\RR$ for which the expressions are finite, the logarithmic Sobolev inequality holds:
\[
    \Ent_\mu(g^2)\le\frac{2}{\alpha}\cdot\E\,\norm{\nabla g(X)}_2^2.
\]
\item[(ii)] \textnormal{(Poincar\'e inequality~\cite[Corollary~4.8.2]{BakryGentilLedoux2014})}
For every smooth $g:\RR^n\to\RR$ for which the expressions are finite,
\[
    \Var(g(X))\le\frac{1}{\alpha}\cdot\E\,\norm{\nabla g(X)}_2^2.
\]
\item[(iii)] \textnormal{(Subgaussianity~\cite[Eq~(5.4.1)]{BakryGentilLedoux2014})}
The random vector $X$ is $\alpha^{-1/2}$-subgaussian:
\[
    \E\,e^{\inner{X}{u}}\le e^{\norm{u}_2^2/(2\alpha)},
    \quad\text{for every }u\in\RR^n.
\]
\item[(iv)] \textnormal{(Moment bounds)}
For every $u\in\RR^n$ with $\norm{u}_2=1$,
\[
    \left(\E\,\abs{\inner{X}{u}}^p\right)^{1/p}\le\sqrt{\frac{p}{\alpha}},
    \quad\text{for }p\ge1.
\]
Moreover,
\[
    \E\,\abs{\inner{X}{u}}^{2p}\le2\left(\frac{2}{\alpha}\right)^p\cdot p!,
    \quad\text{for every integer }p\ge1.
\]
\item[(v)] \textnormal{(Harg\'e's theorem~\cite{Harge2004})}
If $G\sim N(0,\Id)$ and $K\subseteq\RR^n$ is a convex body with $0\in\intr(K)$, then
\[
    \E\,\norm{X}_K\le\frac{1}{\sqrt{\alpha}}\cdot\E\,\norm{G}_K.
\]
\end{enumerate}
\end{lemma}

\subsection{Lattices}
Recall that a lattice $\Lambda \subseteq \RR^n$ is a discrete subgroup of $\RR^n$.
As explained in the introduction, $\Lambda$ admits a \textit{basis}, a set of linearly independent vectors $b_1, \ldots, b_k \in \RR^n$
satisfying that 
\[
    \Lambda = \spn_{\ZZ}\{b_1, \ldots, b_k\}.
\]
We call $k$ the \textit{rank} of $\Lambda$, and we say that $\Lambda$ has \textit{full rank} when $k = n$. When $\Lambda$ has rank $k$, we identify $\spn_{\RR}\Lambda$ with $\RR^k$ using orthonormal coordinates, and
unless otherwise stated, we will always assume that $\Lambda$ has full rank. A \textit{sublattice} of $\Lambda$ is an additive subgroup
$\Gamma\subseteq\Lambda$. Since $\Lambda$ is discrete, so is $\Gamma$, and thus $\Gamma$ is a lattice. When $\Lambda$ is a full rank lattice, we may define its \textit{determinant} as stated above
through the following formula, which is independent of a chosen basis $b_1, \ldots, b_n$:
\[
    \det \Lambda = \abs{\det(b_1, \ldots, b_n)}.
\]
The determinant of a rank $k$ sublattice $\Gamma \subseteq \Lambda$ is defined similarly (in the natural way) under identifying $\spn_{\RR} \Gamma$ with $\RR^k$.
Given a linear subspace $E\subseteq\RR^n$, the intersection
$\Lambda\cap E$ is always a sublattice of $\Lambda$.
We call $E$ a \textit{lattice subspace} of $\Lambda$ if
$\spn_{\RR}(\Lambda\cap E)=E$. Like the case for intersections with subspaces, the projection $\Pi_E\Lambda$ is an additive subgroup of $E$, however, it
may fail to be discrete and therefore may fail to be a lattice.
The following lemma characterizes when $\Pi_E\Lambda$ is a lattice.
\begin{lemma}{(\cite[Lemma~1.1]{KannanLovasz1988})}\label{lem:2.1}
Let $\Lambda\subseteq\RR^n$ be a full-rank lattice and let
$E\subseteq\RR^n$ be a subspace. The following are equivalent:
\begin{enumerate}
    \item[(i)] $\Pi_E\Lambda$ is a lattice in $E$.
    \item[(ii)] $E^\perp$ is a lattice subspace of $\Lambda$.
\end{enumerate}
\end{lemma}
Given a full-rank lattice $\Lambda\subseteq\RR^n$, a Borel set $F\subseteq\RR^n$ is a \textit{fundamental domain} of $\Lambda$ if
the translates $F+\lambda$ for $\lambda\in\Lambda$ partition $\RR^n$
up to a set of Lebesgue measure zero. Given a basis $b_1, \ldots, b_n$ for $\Lambda$, we may form the following fundamental domain,
often called a \textit{fundamental parallelepiped} of $\Lambda$, given by $F = B [0, 1)^n$, where $B = (b_1, \ldots, b_n)$. It is a well-known
fact and easy to derive from the definition that the volume of every fundamental domain is equal to $\det \Lambda$. 
Given two sublattices $\Gamma \subseteq \Lambda$, not necessarily of the same rank, we
call $\Gamma$ \textit{primitive} (in $\Lambda$) if $\Gamma = \Lambda \cap \spn_{\RR} \Gamma$. Equivalently, if $\lambda \in \Lambda$ satisfies $k\lambda \in \Gamma$ for some positive integer $k$, then $\lambda \in \Gamma$. Primitive sublattices in $\Lambda$ are in one-to-one correspondence
with lattice subspaces of $\Lambda$.
To be precise, if $E$ is a lattice subspace of $\Lambda$, then
$\Lambda\cap E$ is primitive.
Conversely, if $\Gamma$ is primitive, then
$E=\spn_{\RR}\Gamma$ is a lattice subspace and
$\Gamma=\Lambda\cap E$. For a sublattice $\Gamma\subseteq\Lambda$, its \textit{dual lattice} is the following lattice
\[
    \Gamma^*=\{\gamma^*\in\spn_{\RR}\Gamma:
    \inner{\gamma^*}{\gamma}\in\ZZ
    \text{ for every }\gamma\in\Gamma\}.
\]
Clearly, $(T\Lambda)^*=(T^{-1})^{\mathsf T}\Lambda^*$ for every
invertible linear map $T:\RR^n\to\RR^n$.
The dual lattice $\Gamma^*$ has the same rank and span as $\Gamma$, furthermore, we have the following duality properties
\[
    (\Gamma^*)^*=\Gamma,
    \qquad \det\Gamma\cdot\det\Gamma^*=1.
\]
The following lemma expresses the determinant of a full-rank lattice $\Lambda$ in terms of lattice subspaces $E \subseteq \RR^n$ of $\Lambda$. We can also write the dual of such sublattices explicitly.
\begin{lemma}{(\cite[Section~2]{RegevStephensDavidowitz2024})}\label{lem:2.2}
Let $\Lambda\subseteq\RR^n$ be a full-rank lattice and let
$E\subseteq\RR^n$ be a lattice subspace of $\Lambda$. Then
\[
    \det\Lambda=\det(\Lambda\cap E)\cdot\det(\Pi_{E^\perp}\Lambda).
\]
Moreover,
\[
    (\Pi_{E^\perp}\Lambda)^*=\Lambda^*\cap E^\perp.
\]
In particular, $E^\perp$ is a lattice subspace of $\Lambda^*$.
\end{lemma}

Recall that a lattice $\Lambda \subseteq \RR^n$ is \textit{stable} if $\det \Lambda = 1$ and every non-zero sublattice
$\{0\} \neq \Gamma \subseteq \Lambda$ satisfies $\det(\Gamma) \ge 1$. Stable lattices will be our primary object of study in this paper, and we
tend to think of them as being "round" lattices.
\begin{lemma}{(\cite[Prop~2.5]{RegevStephensDavidowitz2024})}\label{lem:2.3}
    A full rank lattice $\Lambda \subseteq \RR^n$ is stable if and only if $\Lambda^*$ is stable. If $\Lambda$ is stable and $E$ is a lattice subspace with $\det(\Lambda \cap E) = 1$, then $\Lambda \cap E$ and $\Pi_{E^{\perp}}\Lambda$ are both stable lattices.
\end{lemma}
For any lattice subspace $E$ of a stable lattice $\Lambda$,
the projection $\Pi_{E^\perp}\Lambda$ satisfies $\det (\Pi_{E^\perp}\Lambda) \le1$
and need not be stable. However, every nonzero sublattice $\Gamma \subseteq (\Pi_{E^\perp}\Lambda)^*=\Lambda^*\cap E^\perp$ has determinant $\det \Gamma \ge 1$,
as follows immediately from \cref{lem:2.2,lem:2.3}.

\subsection{Discrete Gaussians}
The following bounds \cref{lem:RhoBoundInTermsOfDet} are immediate consequences of Poisson
summation formula, we omit the proof.
\begin{lemma} \label{lem:RhoBoundInTermsOfDet}
  For any full rank lattice $\Lambda \subseteq \RR^n$ and $s>0$
  one has 
  \[
    \frac{s^n}{\det\Lambda} < \rho_s(\Lambda) \le \max\{1,s^n\}\rho(\Lambda).
  \]
\end{lemma}

\subsection{Quotients and other analytic background}
Let $\Lambda\subseteq\RR^n$ be a full-rank lattice. We define $\RR^n / \Lambda$ as the following set $\RR^n / \Lambda = \{x + \Lambda : x \in \RR^n\}$.
Translations by lattice vectors preserve the Euclidean metric, so it
induces a Riemannian metric on $\RR^n/\Lambda$ for which the quotient map
$\pi:\RR^n\to\RR^n/\Lambda$, $\pi(x)=x+\Lambda$, is a local isometry.
With this metric, $\RR^n/\Lambda$ is a compact flat $n$-dimensional
Riemannian manifold, called a \textit{flat torus}. A function $f:\RR^n\to\RR$ is \textit{$\Lambda$-periodic} if, for every $x\in\RR^n$ and $\lambda\in\Lambda$,
\[
    f(x+\lambda)=f(x).
\]
Such a function is constant on each coset $x+\Lambda$ and therefore
defines a function $\bar f:\RR^n/\Lambda\to\RR$ by
$\bar f(x+\Lambda)=f(x)$.
Conversely, every function $\bar f : \RR^n / \Lambda \to \RR$ induces a
$\Lambda$-periodic function $f=\bar f\circ\pi$. From now on, fix a full-rank lattice $\Lambda\subseteq\RR^n$ and write
$\pi(x)=x+\Lambda$.
Choose a fundamental domain $F \subseteq \RR^n$ for $\Lambda$. The \emph{uniform probability measure} on $\RR^n/\Lambda$ is the law of
$U+\Lambda$, where $U\sim\Unif(F)$ is uniform. It is independent of the choice of $F$ and
satisfies
\[
    \E_{U\sim\Unif(F)}\,f(U+\Lambda)
    =\frac{1}{\det\Lambda}\int_F f(x+\Lambda)\,dx
\]
for every bounded measurable $f:\RR^n/\Lambda\to\RR$. More generally, any probability measure $\mu$ on $\RR^n$ induces a law
on the quotient. This \emph{pushforward} $\pi_{\#}\mu$ is the law of
$X+\Lambda$ when $X\sim\mu$; equivalently,
\[
    (\pi_{\#}\mu)(A)=\mu(\pi^{-1}(A)),
    \quad\text{for Borel }A\subseteq\RR^n/\Lambda.
\]
We can define differentiation on $\RR^n/\Lambda$ (it is in the standard way as we would define for any Riemannian manifold).
A function $f:\RR^n/\Lambda\to\RR$ is smooth if $f\circ\pi$ is smooth
on $\RR^n$, and its gradient is
\[
    \nabla f(x+\Lambda)=\nabla(f\circ\pi)(x).
\]
This is well defined because $\nabla(f\circ\pi)$ is $\Lambda$-periodic. Let $\mu$ be a Borel probability measure on $\RR^n$ or $\RR^n/\Lambda$. For a nonnegative Borel function $h$ for which the following expressions are finite, its \emph{entropy} is
\[
    \Ent_\mu(h)
    =\int h\log h\,d\mu
    -\left(\int h\,d\mu\right)\log\left(\int h\,d\mu\right),
\]
with the convention $0\log0=0$. The \emph{logarithmic Sobolev constant} $C_{\LS}(\mu)$ is the
least $C\ge0$ such that
\[
    \Ent_\mu(f^2)\le 2C\int\norm{\nabla f}_2^2\,d\mu
\]
for every smooth function $f$ for which the expressions are finite.
We set $C_{\LS}(\mu)=+\infty$ if no finite $C$ exists. The following criterion, explained in the proof outline of \cref{thm:1.1}, is an easy way to construct a measure on $\RR^n$ whose projection to $\RR^n / \Lambda$ is uniform.

\begin{lemma}\label{lem:2.5}
Let $X$ have probability density $p$ on $\RR^n$. Then $X+\Lambda$ is
uniformly distributed on $\RR^n/\Lambda$ if and only if
\[
    \sum_{\lambda\in\Lambda}p(x+\lambda)=\frac{1}{\det\Lambda},
    \quad\text{for almost every }x\in\RR^n.
\]
\end{lemma}
\begin{proof}
Choose a basis $b_1, \ldots, b_n$ for $\Lambda$ and form the matrix $B = (b_1, \ldots, b_n)$. Set $F=B[0,1)^n$.
For every $x \in \RR^n$, let $r(x)\in F$ be the unique point in $F$ satisfying $x - r(x) \in \Lambda$.
For every Borel set $A\subseteq F$, we may use the disjointness of the translates $F + \lambda$ for $\lambda \in \Lambda$
and Tonelli's theorem to get
\begin{align*}
    \PP(r(X)\in A)
    &=\sum_{\lambda\in\Lambda}\PP(X\in A+\lambda)\\
    &=\sum_{\lambda\in\Lambda}\int_{A+\lambda}p(y)\,dy
      =\sum_{\lambda\in\Lambda}\int_A p(x+\lambda)\,dx\\
    &=\int_A\sum_{\lambda\in\Lambda}p(x+\lambda)\,dx.
\end{align*}
Since $\Vol_n(F)=\det\Lambda$, the random vector $r(X)$ is uniform
on $F$ precisely when
\[
    \sum_{\lambda\in\Lambda}p(x+\lambda)=\frac{1}{\det\Lambda},
    \quad\text{for almost every }x\in F.
\]
The latter holds for almost every $x \in F$ if and only if it holds for almost every $x \in \RR^n$ by periodicity, and since $r(X)$ is uniform on $F$ if and only if $X + \Lambda$ is uniform on $\RR^n/\Lambda$, we are done.
\end{proof}

\section{Proving \texorpdfstring{\cref{thm:1.1}}{Theorem \getrefnumber{thm:1.1}}}\label{sec:3}

\subsection{Overview of the proof of \texorpdfstring{\cref{thm:1.1}}{Theorem \getrefnumber{thm:1.1}}}

The main ideas behind the proof of \cref{thm:1.1} can be described briefly,
although the details occupy the rest of this section. For a symmetric positive-definite matrix $A \succ 0$ we use the Gaussian weight
\[
    \rho_A(x)=e^{-\pi\inner{x}{A^{-1}x}},
\]
and normalize $\rho_A$ by dividing by sum of $\rho_A$ over
$x+\Lambda$ and by $\det\Lambda$:
\[
    f(x)=\frac{1}{\det \Lambda} \cdot \frac{\rho_A(x)}{\rho_A(x+\Lambda)}.
\]
This immediately implies that $\sum_{\lambda\in\Lambda}f(x+\lambda)=1/\det\Lambda$ for every $x \in \RR^n$.
Further, observe that $f$ is a probability density over $\RR^n$. Also, $f$ is smooth, strictly positive, and symmetric, and its law is
uniform modulo $\Lambda$ by \cref{lem:2.5}. Thus it remains to choose $A$ so that $f$ is $(2\pi/e)$-strongly log-concave.
For $x \in \RR^n$, we denote by $D_{x + \Lambda,A}$ the \textit{discrete Gaussian} defined through the following law on $x + \Lambda$:
\[
    \PP_{Y \sim D_{x + \Lambda, A}}(Y = y) = \frac{\rho_A(y)}{\rho_A(x + \Lambda)}, \quad \text{for } y \in x + \Lambda.
\]
Differentiating $-\log f$ twice yields, by \cref{lemma:3.3},
\begin{equation}\label{eq:1.2}
    \nabla^2(-\log f)(x)
    =4\pi^2 A^{-1}\Cov(D_{x+\Lambda,A})A^{-1}.
\end{equation}
Regev and Stephens-Davidowitz~\cite[Corollary~3.2]{RegevStephensDavidowitz2017}
show that
\begin{equation}\label{eq:1.3}
    \Cov(D_{x+\Lambda,A})
    \succcurlyeq\Cov(D_{\Lambda,A}),
    \quad\text{for every }x\in\RR^n.
\end{equation}
Plugging \eqref{eq:1.3} into \eqref{eq:1.2} yields
\begin{equation} \label{eq:HessianVsInvACovInvA}
    \nabla^2(-\log f)(x)
    \succcurlyeq4\pi^2 A^{-1}\Cov(D_{\Lambda,A})A^{-1}.
  \end{equation}
  It therefore suffices to find $A \succ 0$ satisfying
  \begin{equation} \label{eq:OverviewAinvCovAinvvsID}
    A^{-1}\Cov(D_{\Lambda,A})A^{-1} \succcurlyeq \frac{1}{2\pi e} \Id.
  \end{equation}
  For that purpose we consider the optimization problem
  \begin{equation} \label{eq:MaxProblem}
   \max\Big\{ \frac{\log\big( \rho_{A}(\Lambda)\big)}{\tr(A)} \mid A \succcurlyeq 0, A \neq 0 \Big\}
 \end{equation}
which is a variant of a quantity introduced by Dadush and Regev~\cite[Def 5.1]{DadushRegev2016}.
Testing $A=e\Id$ and applying Poisson summation gives an objective value greater than $1/(2e)$, since $\det\Lambda\le1$. We show below that a maximizer exists, allowing singular matrices with the Gaussian weight supported on their image. At such a maximizer, the optimality condition forces $E=\im(A)$ to be a nonzero lattice subspace and yields \eqref{eq:OverviewAinvCovAinvvsID} on $E$, with $A_E=A|_E$ and the lattice $\Lambda\cap E$.

Hence using $A_E$ gives us a $(2\pi/e)$-strongly log-concave density $f_{E}: E \to \RR$ on $E$, which is uniform modulo
$\Lambda\cap E$. If $E=\RR^n$, we are done. Otherwise,
the dual of $\Pi_{E^\perp}\Lambda$ is $\Lambda^*\cap E^\perp$ and is therefore
weakly stable. Induction on the dimension gives a $(2\pi/e)$-strongly log-concave density
$f_{E^{\perp}} : E^{\perp} \to \RR$ on $E^\perp$ that is uniform modulo $\Pi_{E^\perp}\Lambda$.
Then $f(x+y)=f_E(x)\cdot f_{E^{\perp}}(y)$,
for $x\in E$ and $y\in E^\perp$, is $(2\pi/e)$-strongly log-concave
and uniform modulo $\Lambda$. We now give the details.

\subsection{Proof of \texorpdfstring{\cref{thm:1.1}}{Theorem \getrefnumber{thm:1.1}}}

Our first inequality \cref{lemma:3.1} that we present is a slight modification of a result due to Regev and Stephens-Davidowitz~\cite{RegevStephensDavidowitz2017}. First, observe as an immediate consequence of the
parallelogram law that
\[
    \rho_A(x)^2\rho_A(y)^2 = \rho_A(x + y)\rho_A(x - y), \quad \text{for every }x,y\in\RR^n.
\]
We provide a "lattice analogue" of this formula.
\begin{lemma} \label{lemma:3.1}
    For every positive-definite symmetric matrix $A \succ 0$ and every $x, y \in \RR^n$, we have that
    \[
        \rho_A(x + \Lambda)^2\rho_A(y + \Lambda)^2 \le \rho_A(\Lambda)^2 \rho_A(x + y + \Lambda)\rho_A(x - y + \Lambda).
    \]
\end{lemma}
\begin{proof}
    The parallelogram law yields that
    \begin{align*}
    \rho_A(x + \Lambda)\rho_A(y + \Lambda) &= \sum_{\lambda_1, \lambda_2 \in \Lambda} \rho_A(\lambda_1 + x)\rho_A(\lambda_2 + y) \\
        &= \sum_{\lambda_1, \lambda_2 \in \Lambda} \sqrt{\rho_A(\lambda_1 + \lambda_2 + x + y)}\sqrt{\rho_A(\lambda_1 - \lambda_2 + x - y)}.
    \end{align*}
    Now we observe that the change of variables $(\lambda_1, \lambda_2) \to (\lambda_1 + \lambda_2, \lambda_1 - \lambda_2)$ maps $\Lambda^2$ bijectively onto the pairs of lattice vectors congruent modulo $2\Lambda$. Thus,
    \begin{align*}
        \rho_A(x + \Lambda)\rho_A(y + \Lambda)
        &= \sum_{\substack{\lambda_1, \lambda_2 \in \Lambda\\
            \lambda_1 \equiv \lambda_2 \pmod{2\Lambda}}}
            \sqrt{\rho_A(\lambda_1 + x + y)}
            \sqrt{\rho_A(\lambda_2 + x - y)}\\
        &= \sum_{\lambda+2\Lambda\in\Lambda/(2\Lambda)}
            \left(\sum_{\lambda_1\in\lambda+2\Lambda}
                \sqrt{\rho_A(\lambda_1 + x + y)}\right)
            \left(\sum_{\lambda_2\in\lambda+2\Lambda}
                \sqrt{\rho_A(\lambda_2 + x - y)}\right).
    \end{align*}
    Applying Cauchy--Schwarz and using the parallelogram law again gives
    \[
    \begin{aligned}
        &\rho_A(x + \Lambda)^2\rho_A(y + \Lambda)^2\\
        &\le
            \left[\sum_{\lambda+2\Lambda\in\Lambda/(2\Lambda)}
                \left(\sum_{\lambda_1\in\lambda+2\Lambda}
                    \sqrt{\rho_A(\lambda_1 + x + y)}\right)^2\right]\\
        &\qquad {}\times
            \left[\sum_{\lambda+2\Lambda\in\Lambda/(2\Lambda)}
                \left(\sum_{\lambda_2\in\lambda+2\Lambda}
                    \sqrt{\rho_A(\lambda_2 + x - y)}\right)^2\right]\\
        &= \rho_A(\Lambda)^2
            \rho_A(x + y + \Lambda)\rho_A(x - y + \Lambda),
    \end{aligned}
    \]
    as desired.
  \end{proof}
Regev and Stephens-Davidowitz~\cite{RegevStephensDavidowitz2017} prove the inequality in Lemma~\ref{lemma:3.1} for the case that $A = \Id$.
As a consequence of \cref{lemma:3.1} we obtain \cref{lemma:3.2}, which gives properties about the map $x \to \Cov(D_{x + \Lambda, A})$.

\begin{lemma} \label{lemma:3.2}
    Let $\Lambda\subseteq\RR^n$ be a full-rank lattice, let $x\in\RR^n$ and 
    let $A$ be a symmetric positive-definite matrix. Then the covariance of $D_{x+\Lambda,A}$ is minimized when $x=0$,
    i.e.
    \[
       0 \preccurlyeq \Cov(D_{\Lambda,A})
        \preccurlyeq\Cov(D_{x+\Lambda,A}).
    \]
\end{lemma}
\begin{proof}
    The law $D_{\Lambda,A}$ is symmetric, so its mean is zero.
    Fix $x\in\RR^n\setminus\{0\}$. Since $\Lambda$ spans $\RR^n$, there
    is some $\lambda\in\Lambda$ with $\inner{x}{\lambda}\neq0$;
    otherwise $x$ would be orthogonal to all of $\RR^n$, forcing $x=0$.
    Every lattice point has positive probability, and hence
    \[
    \begin{aligned}
        \inner{x}{\Cov(D_{\Lambda,A})x}
        &=\E_{Y\sim D_{\Lambda,A}}\,\inner{x}{Y}^2\\
        &\ge\PP_{Y\sim D_{\Lambda,A}}(Y=\lambda)
            \inner{x}{\lambda}^2>0.
    \end{aligned}
    \]
    Since $x\neq0$ was arbitrary, this proves that $\Cov(D_{\Lambda,A})$ is positive-definite.
    
    Next, we prove that the covariance
    of $D_{x + \Lambda,A}$ is minimized when $x = 0$.
    For the comparison under shifts, taking logarithms in
    \cref{lemma:3.1} gives
    \[
    \begin{aligned}
        \log\rho_A(x+y+\Lambda)+\log\rho_A(x-y+\Lambda) -2\log\rho_A(x+\Lambda)-2\log\rho_A(y+\Lambda)
        +2\log\rho_A(\Lambda)\ge0.
    \end{aligned}
    \]
    For fixed $x\in\RR^n$, the expression above has a minimum at $y=0$.
    Since the maps $y\mapsto x+y$ and $y\mapsto x-y$ are affine,
    its Hessian at zero is
    \[
        2\nabla_x^2\log\rho_A(x+\Lambda)
        -2\left.\nabla_z^2\log\rho_A(z+\Lambda)\right|_{z=0}
        \succcurlyeq0.
    \]
    Substituting the Hessian identity proved in \cref{lemma:3.3} below
    cancels the terms $-2\pi A^{-1}$. Congruence by $A$ then gives
    \[
        \Cov(D_{x+\Lambda,A})
        \succcurlyeq\Cov(D_{\Lambda,A}). \qedhere
    \]
\end{proof}
Again, the statement of Lemma~\ref{lemma:3.2} is known in the case of $A = \Id$, see the remark on page 5 of \cite{RegevStephensDavidowitz2017}. 
Recall our candidate density $f$ given by the equation
\[
    f(x) = \frac{1}{\det \Lambda} \cdot \frac{\rho_A(x)}{\rho_A(x + \Lambda)}.
\]
The following \cref{lemma:3.3} relates the Hessian of $-\log f$
to the covariance of the discrete Gaussian $D_{x+\Lambda, A}$.
\begin{lemma} \label{lemma:3.3}
    Let $\Lambda\subseteq\RR^n$ be a full-rank lattice, let $A\succ0$ be
    symmetric, and let $f$ be the density defined above.
    Then
    \[
    \begin{aligned}
        \nabla^2(-\log f)(x)
        &=4\pi^2A^{-1}\Cov(D_{x+\Lambda,A})A^{-1}, \quad\text{for every }x\in\RR^n.
    \end{aligned}
    \]
\end{lemma}
\begin{proof}
    We may take logarithms giving
    \[
        -\log f(x)=\log\det\Lambda
        +\pi\inner{x}{A^{-1}x}+\log\rho_A(x+\Lambda).
    \]
    Hence
    \[
        \nabla^2(-\log f)(x)
        =2\pi A^{-1}+\nabla_x^2\log\rho_A(x+\Lambda).
    \]
    Termwise differentiation of $\rho_A(x + \Lambda)$ yields
    \[
        \nabla_x^2\log\rho_A(x+\Lambda)
        =-2\pi A^{-1}
        +4\pi^2A^{-1}\Cov(D_{x+\Lambda,A})A^{-1}. \qedhere
    \]
\end{proof}

\color{black}

\begin{lemma} \label{lem:InductionOfMainTheorem}
  Let $\Lambda \subseteq \RR^n$ be a non-zero lattice with $\det(\Lambda) \le 1$.
  Then there exists a matrix $A \in \RR^{n \times n}$ with $A \succcurlyeq 0$ so that $E = \im(A)$ is a nonzero lattice subspace with
  \begin{equation} \label{eq:AinvCovAInvAtLeastId}
    A^{+}\Cov(D_{\Lambda \cap E,A})A^+ \succcurlyeq \frac{1}{2\pi e} \Id_E
  \end{equation}
\end{lemma}
\begin{proof}
  Without loss of generality we may assume that $\Lambda$ has full rank.
  We use the extension
  \[
    \rho_A(x) = \begin{cases} \exp(-\pi \inner{x}{A^+x}) & \text{if }x \in \im(A) \\
      0 & \text{if }x \notin \im(A) \end{cases}
  \]
  for (possibly singular) matrices $A \succcurlyeq 0$. 
  We define the potential function 
  \[
  \Phi_{\Lambda}(A) =  \frac{\log \rho_{A}(\Lambda)}{\tr(A)}
  \]
  and consider the optimization problem
  \begin{equation} \label{eq:MaxProblemII}
 \Phi^*_{\Lambda} = \max\Big\{ \Phi_{\Lambda}(A) \mid A \succcurlyeq 0, A \neq 0 \Big\}
 \end{equation}
 which can be understood as a continuous variant of finding a dense sublattice. 

 First, we lower bound the value of \eqref{eq:MaxProblemII}.
 In fact, testing $A=e\Id$ gives
\[
 \Phi_{\Lambda}(e\Id)
 =\frac{\log\rho_{e\Id}(\Lambda)}{\tr(e\Id)}
 >\frac{\log(e^{n/2}/\det\Lambda)}{en}
 \ge\frac{1}{2e},
\]
where we use Lemma~\ref{lem:RhoBoundInTermsOfDet} at scalar width
$s=\sqrt e$, the identity $\rho_{\sqrt e}=\rho_{e\Id}$, and
the assumption $\det\Lambda\le1$. We now make a claim.

{\bf Claim I.} \emph{The maximum in \eqref{eq:MaxProblemII} is attained.} \\
{\bf Proof of Claim I.} For each fixed $x \in \RR^n$ and every $A \succcurlyeq 0$, write
\[
  q_A(x) = \begin{cases} \inner{x}{A^+x} & \text{if }x \in \im(A), \\
    +\infty & \text{if }x \notin \im(A).
  \end{cases}
\]
The variational identity
\[
  q_A(x) = \sup_{z \in \RR^n}\big\{2\inner{x}{z}-\inner{z}{Az}\big\}
\]
shows that $q_A(x)$ is lower semicontinuous in $A$, since it is a supremum of continuous functions of $A$. Hence $\rho_A(x)=e^{-\pi q_A(x)}$, with $e^{-\infty}=0$, is upper semicontinuous.

To pass to the lattice sum, suppose that $A_j \to A \neq 0$, where all these matrices are positive semidefinite. Choose $u>0$ such that $A_j \preccurlyeq u\Id$ for every $j$. For every $\lambda\in\Lambda$ we then have
\[
  0\le \rho_{A_j}(\lambda)\le e^{-\pi\norm{\lambda}_2^2/u}.
\]
The expression on the right is summable over $\Lambda$. Applying Fatou's lemma to the difference between this dominating Gaussian and $\rho_{A_j}$ gives
\[
  \limsup_{j\to\infty}\rho_{A_j}(\Lambda)
  \le\sum_{\lambda\in\Lambda}\limsup_{j\to\infty}\rho_{A_j}(\lambda)
  \le\rho_A(\Lambda).
\]
Since $\tr(A_j)\to\tr(A)>0$, this proves that $\Phi_\Lambda$ is upper semicontinuous on the nonzero positive semidefinite matrices, including singular ones.

It remains to restrict the optimization to a compact set. For $t>0$, define
\[
  \Phi_{\Lambda,t}^*
  =\max\left\{\frac{\log\rho_A(\Lambda)}{t}
  \mid A\succcurlyeq0,\ \tr(A)=t\right\}.
\]
This maximum exists because its domain is compact and $\Phi_\Lambda$ is upper semicontinuous there. Let
\[
  m=\min_{\lambda\in\Lambda\setminus\{0\}}\norm{\lambda}_2>0,
  \qquad
  C=\sum_{\lambda\in\Lambda\setminus\{0\}}e^{-\pi\norm{\lambda}_2^2/2}<\infty.
\]
If $0<t\le1$ and $\tr(A)=t$, then $A\preccurlyeq t\Id$, and therefore
\[
\begin{split}
  \rho_A(\Lambda)-1
  &\le\sum_{\lambda\in\Lambda\setminus\{0\}}e^{-\pi\norm{\lambda}_2^2/t}\\
  &\le e^{-\pi m^2/(2t)}
     \sum_{\lambda\in\Lambda\setminus\{0\}}e^{-\pi\norm{\lambda}_2^2/(2t)}
  \le C e^{-\pi m^2/(2t)}.
\end{split}
\]
Using $\log(1+r)\le r$ for $r\ge0$, we obtain
\[
  0\le\Phi_{\Lambda,t}^*\le \frac{C}{t}e^{-\pi m^2/(2t)}\longrightarrow0
  \qquad\text{as }t\to0^+.
\]
For $t\ge1$, Lemma~\ref{lem:RhoBoundInTermsOfDet} gives
\[
  \rho_A(\Lambda)\le\rho_{t\Id}(\Lambda)=\rho_{\sqrt t}(\Lambda)\le t^{n/2}\rho(\Lambda),
\]
and hence
\[
  0\le\Phi_{\Lambda,t}^*
  \le\frac{\frac n2\log t+\log\rho(\Lambda)}{t}
  \longrightarrow0
  \qquad\text{as }t\to\infty.
\]
The candidate $A=e\Id$ has $\Phi_\Lambda(e\Id)>1/(2e)$. Thus there exist $0<\ell<en<u$ such that every matrix with $\tr(A)\notin[\ell,u]$ has objective value smaller than $\Phi_\Lambda(e\Id)$. Consequently,
\[
  \Phi_\Lambda^*
  =\max\left\{\Phi_\Lambda(A)\mid A\succcurlyeq0,\ \ell\le\tr(A)\le u\right\}.
\]
The domain is compact and avoids the zero matrix. Upper semicontinuity therefore proves that the maximum is attained. \qed

Now fix $A$ as a matrix attaining \eqref{eq:MaxProblemII} and denote $E = \im(A)$.
Let $H$ be a symmetric matrix with $\im(H)\subseteq E$.
For $\abs{t}$ sufficiently small, $A+tH$ is positive definite on $E$ and
vanishes on $E^\perp$, so its Gaussian weight has the fixed support
$\Lambda\cap E$. The differentiated series is locally dominated by
a summable Gaussian times a polynomial. Since $D_{\Lambda\cap E,A}$
is centered, termwise differentiation gives
\[
    \left.\frac{d}{dt}\log\rho_{A+tH}(\Lambda)\right|_{t=0}
    =\pi\,\tr\bigl(A^+\Cov(D_{\Lambda\cap E,A})A^+H\bigr).
\]
Applying the quotient rule and optimality of $A$, we obtain
\[
\begin{aligned}
    0
    &=\left.\frac{d}{dt}\Phi_\Lambda(A+tH)\right|_{t=0}\\
    &=\frac{1}{\tr(A)}\,
      \tr\!\left[\left(\pi A^+\Cov(D_{\Lambda\cap E,A})A^+
          -\Phi_\Lambda(A)\Id_E\right)H\right].
\end{aligned}
\]
As this holds for every such $H$ we must indeed have
\begin{equation} \label{eq:MainLemmaConclusion} 
 A^+\Cov(D_{\Lambda \cap E,A})A^+ = \frac{\Phi_{\Lambda}(A)}{\pi} \cdot \Id_E
\end{equation}
From this we can see that $E$ must be a lattice subspace since otherwise $\im(\Cov(D_{\Lambda \cap E,A}))$ would be a strict subspace of $E$. 
Finally, as we have proven earlier, $\Phi_{\Lambda}(A) \ge \frac{1}{2e}$ which combined with \eqref{eq:MainLemmaConclusion} concludes the proof.
\end{proof}
We are finally ready to prove \cref{thm:1.1}.
\begin{proof}[Proof of \cref{thm:1.1}]
  We prove the existence statement by induction on $n$.
  Since $\Lambda^*$ is weakly stable we have in particular that $\det \Lambda = \frac{1}{\det \Lambda^*} \le 1$. Hence we may apply Lemma~\ref{lem:InductionOfMainTheorem} and obtain a matrix $A \in \RR^{n \times n}$
  satisfying \eqref{eq:AinvCovAInvAtLeastId} while the image of $A$ is a lattice subspace $E$.
    Define
    \[
        f_E(x)=\frac{1}{\det(\Lambda\cap E)}
        \cdot\frac{\rho_{A}(x)}{\rho_{A}(x+(\Lambda\cap E))},
        \quad\text{for }x\in E.
    \]
    This function is smooth, strictly positive, and even. Since the
    denominator is $(\Lambda\cap E)$-periodic,
    \[
        \sum_{\lambda\in\Lambda\cap E}f_E(x+\lambda)
        =\frac{1}{\det(\Lambda\cap E)}.
    \]
    Integrating over a fundamental domain of $\Lambda\cap E$ gives
    $\int_E f_E=1$, and \cref{lem:2.5} shows that its law is uniform
    modulo $\Lambda\cap E$. Applying
    \cref{lemma:3.3}, then \cref{lemma:3.2}, and finally \cref{lem:InductionOfMainTheorem} gives
    \[
    \begin{aligned}
        \nabla_E^2(-\log f_E)(x)
        &=4\pi^2A^{+}
          \Cov(D_{x+(\Lambda\cap E),A})A^{+}\\
        &\succcurlyeq4\pi^2A^{+}
          \Cov(D_{\Lambda\cap E,A})A^+\\
        &\succcurlyeq \frac{2\pi}{e}\Id_E.
    \end{aligned}
    \]
    If $E=\RR^n$, taking $f=f_E$ proves the existence statement.
    Since $E\neq\{0\}$ by \cref{lem:InductionOfMainTheorem}, this also proves the base
    case $n=1$.

    Suppose now that $0<\dim E<n$.
    As $E$ is a lattice subspace,  $\Pi_{E^\perp}\Lambda$ is a full-rank lattice in
    $E^\perp$, and \cref{lem:2.2} gives
    \[
        (\Pi_{E^\perp}\Lambda)^*=\Lambda^*\cap E^\perp.
    \]
    Every nonzero sublattice of $\Lambda^*\cap E^\perp$ is a sublattice of $\Lambda^*$,
    so $(\Pi_{E^\perp}\Lambda)^*$ is weakly stable. Since $\dim E^\perp<n$, the induction
    hypothesis gives a smooth, strictly positive, even probability
    density $f_{E^\perp}$ satisfying, for every $y\in E^\perp$,
    \[
        \nabla_{E^\perp}^2(-\log f_{E^\perp})(y)
        \succcurlyeq\frac{2\pi}{e}\Id_{E^\perp},
        \qquad
        \sum_{\lambda\in\Pi_{E^\perp}\Lambda}f_{E^\perp}(y+\lambda)=\frac{1}{\det(\Pi_{E^\perp}\Lambda)}.
    \]
    Define
    \[
        f(x+y)=f_E(x)\cdot f_{E^\perp}(y),
        \quad\text{for }x\in E\text{ and }y\in E^\perp.
    \]
    By the orthogonal decomposition $\RR^n=E\oplus E^\perp$, this is a
    smooth, strictly positive, even probability density, and
    \[
        \nabla^2(-\log f)(x+y)
        =\nabla_E^2(-\log f_E)(x)
          \oplus\nabla_{E^\perp}^2(-\log f_{E^\perp})(y)
        \succcurlyeq\frac{2\pi}{e}\Id.
    \]

    It remains to check uniformity modulo $\Lambda$ using \cref{lem:2.5}.
    For each $\lambda_1\in\Pi_{E^\perp}\Lambda$, choose $s_{\lambda_1}\in E$ such that
    $s_{\lambda_1}+\lambda_1\in\Lambda$. The fibers of $\Pi_{E^\perp}$ give
    \[
        \Lambda=\bigsqcup_{\lambda_1\in\Pi_{E^\perp}\Lambda}
        \bigl(s_{\lambda_1}+\lambda_1+(\Lambda\cap E)\bigr).
    \]
    The two identities above and the determinant formula
    in \cref{lem:2.2} therefore imply
    \[
    \begin{aligned}
        \sum_{\lambda\in\Lambda}f(x+y+\lambda)
        &=\sum_{\lambda_1\in\Pi_{E^\perp}\Lambda}f_{E^\perp}(y+\lambda_1)
          \sum_{\lambda_2\in\Lambda\cap E}f_E(x+s_{\lambda_1}+\lambda_2)\\
        &=\frac{1}{\det(\Lambda\cap E)}
          \sum_{\lambda_1\in\Pi_{E^\perp}\Lambda}f_{E^\perp}(y+\lambda_1)\\
        &=\frac{1}{\det(\Lambda\cap E)\cdot\det(\Pi_{E^\perp}\Lambda)}
          =\frac{1}{\det\Lambda}.
    \end{aligned}
    \]
    By \cref{lem:2.5}, the law with density $f$ is uniform modulo
    $\Lambda$, completing the induction.

    For the converse, let $f$ satisfy the stated properties with "curvature"
    $\alpha>0$. Fix a nonzero lattice subspace $E$ of $\Lambda^*$,
    and set $k=\dim E$. By \cref{lem:2.1,lem:2.2},
    $\Pi_E\Lambda$ is a lattice with $(\Pi_E\Lambda)^*=\Lambda^*\cap E$. The marginal
    $f_E(x)=\int_{E^\perp}f(x+y)\,dy$ is even and
    $\alpha$-strongly log-concave by
    Pr\'ekopa's theorem~\cite[Theorem~6]{Prekopa1973}.
    Its law is uniform modulo $\Pi_E\Lambda$, so \cref{lem:2.5} and continuity
    give $f_E(0)\le1/\det(\Pi_E\Lambda)$. Hence
    \[
        1=\int_E f_E(x)\,dx
        \le f_E(0)\int_E e^{-\alpha\norm{x}_2^2/2}\,dx
        \le\det(\Lambda^*\cap E)
          \left(\frac{2\pi}{\alpha}\right)^{k/2}.
    \]
    Thus every primitive sublattice of
    $\sqrt{2\pi/\alpha}\,\Lambda^*$ has determinant at least one, and thus so does every nonzero sublattice. This proves weak stability.
\end{proof}

\section{Applications}\label{sec:4}
As highlighted, we intend to only list a couple implications of the log-concave lattice lift, in order to keep the paper sleek. Our perspective highlights the computer science applications. We would very much like to see future applications between lattices and high-dimensional convex geometry.
\subsection{The covering radius}

We first prove \cref{thm:CoveringRadius}, which we restate for convenience.
\begin{theorem*}[Theorem~\ref{thm:CoveringRadius}]
Let $\Lambda\subseteq\RR^n$ be a lattice such that $\Lambda^*$ is weakly stable and let $K\subseteq\RR^n$
be a convex body with $0 \in\intr(K)$. Then
\[
    \mu(\Lambda,K)\le\sqrt{\frac{2e}{\pi}}\cdot 
    \E\,\norm{G}_K,
\]
where $G \sim N(0,\Id)$ is a standard Gaussian.
\end{theorem*}
\begin{proof}
Let $X$ have the even density $f$ from \cref{thm:1.1}, and set
\[
    \delta(x)=\min_{\lambda\in\Lambda}\norm{x-\lambda}_K.
\]
Fix $x\in\RR^n$. For every $y\in\RR^n$, subadditivity gives
$\delta(x)\le\delta(y)+\delta(x-y)$.
Since $\delta$ is $\Lambda$-periodic and both
$X+\Lambda$ and $x-X+\Lambda$ are uniform, we have
\[
    \E\,\delta(x-X)=\E\,\delta(X).
\]
Therefore,
\[
    \delta(x)
    \le\E\,\delta(X)+\E\,\delta(x-X)
    =2\E\,\delta(X)
    \le2\E\,\norm{X}_K
    \le\sqrt{\frac{2e}{\pi}}\cdot \E\,\norm{G}_K.
\]
The last inequality is
\cref{lem:ConcentrationForAlphaLogConcave}(v), with $\alpha=2\pi/e$.
Taking the supremum over $x$ proves the claim.
\end{proof}

\subsection{Flatness and subspace flatness}

We combine \cref{thm:CoveringRadius} with the following estimate
\cite[Theorem~1.4]{Bizeul2026} and the filtration argument of
\cite[Sections~3--4]{ReisRothvoss2026}.

\begin{lemma}\label{lem:HalfDimensionalCovering}
For every full-rank lattice $\Lambda\subseteq\RR^n$ and convex body
$K\subseteq\RR^n$, there are $a\in\intr(K)$, a lattice
subspace $E$ of $\Lambda$ with $\dim E\ge n/2$, and a nonzero
subspace $W$ for which $\Pi_W\Lambda$ is a lattice, such that
\[
\mu(\Lambda\cap E,(K-a)\cap E)
\le C\sqrt{\log n}\cdot R_W(\Lambda,K).
\]
\end{lemma}
\begin{proof}
Both covering radii and the ratios $R_W$ are invariant under invertible
linear maps, with subspaces identified through the induced quotient
maps.
Translating $K$ and applying such a map, assume that its uniform
measure has mean zero and covariance $\Id$. We will take $a=0$ in these
coordinates. By
\cite[Theorem~1.4]{Bizeul2026},
\[
\E\,\norm{G}_K
\le C\sqrt{\log n},
\]
where $G \sim N(0, \Id)$. For a $d$-dimensional subspace $W$, let $p_W$ be the density of the
projection onto $W$ of a uniform point of $K$. Since its second moment
is $d$, Markov's inequality and the Rogers and Shephard~\cite[Theorem~1]{RogersShephard1958} inequality give
\[
\norm{p_W}_\infty
\ge\frac{1}{2\Vol_d(\sqrt{2d}\,B_2^d)},
\qquad
\Vol_d(\Pi_WK)\cdot\norm{p_W}_\infty\le\binom nd.
\]
Consequently,
\begin{equation}\label{eq:IsotropicProjectionVolume}
\Vol_d(\Pi_WK)^{1/d}\le C\binom nd^{1/d}\le Cn/d.
\end{equation}
The case $d=n$ follows from the same formulas, with $p_W$ the uniform
density on $K$.

We recall from \cite[Corollary~26]{ReisRothvoss2026}
that (every lattice) $\Lambda$ admits a
\emph{well-separated $2$-stable filtration}
$\{0\}=\Lambda_0\subsetneq\cdots\subsetneq\Lambda_k=\Lambda$. That is, denoting
$E_j=\spn_{\RR}\Lambda_j$, $L_j=\Pi_{E_{j-1}^\perp}\Lambda_j$,
$d_j=\rank L_j$, and $r_j=(\det L_j)^{1/d_j}$,
such a filtration satisfies
$r_j<r_{j+1}$ for $j \le k-1$ and $r_j\le \frac{r_{j+2}}{2}$ for $j \le k-2$, and $2r_jL_j^*$ is weakly stable.
Fix such a filtration and denote
$K_j=\Pi_{E_{j-1}^\perp}(K\cap E_j)$,
$G_j=\Pi_{E_j\cap E_{j-1}^\perp}G$ and $G^\perp_j = \Pi_{(E_j\cap E_{j-1}^\perp)^\perp} G$.
Since $K\cap E_j\cap E_{j-1}^\perp\subseteq K_j$,
we have $\norm{x}_{K_j}\le\norm{x}_K$ for
$x\in E_j\cap E_{j-1}^\perp$. Jensen's inequality gives
\[
\E\norm{G_j}_{K_j}
\le\E\norm{G_j}_K
\le\E\norm{G_j + G_j^\perp}_K
=\E\norm{G}_K
\le C\sqrt{\log(2n)}.
\]
Applying \cref{thm:CoveringRadius} to $L_j/(2r_j)$ therefore gives
\[
\mu(L_j,K_j)\le \sqrt{\frac{2e}{\pi}} \cdot 2r_j\E\,\norm{G}_K.
\]
Let $i$ be minimal with $\dim E_i\ge n/2$, and set
$E=E_i$, $W=E_{i-1}^\perp$, and $d=\dim W>n/2$.
The triangle inequality along with
\cite[Lemma~41]{ReisRothvoss2026} then yields
\begin{equation}\label{eq:MuVsRW1}
\begin{aligned}
\mu(\Lambda_i,K\cap E)
&\le\sum_{j\le i}\mu(L_j,K_j)\\
&\le\sqrt{\frac{2e}{\pi}}\,
     \left(\sum_{j\le i}2r_j\right)\E\,\norm{G}_K\\
&\le\sqrt{\frac{2e}{\pi}}\cdot8r_i\E\,\norm{G}_K
 \le Cr_i\sqrt{\log n}.
\end{aligned}
\end{equation}
Here $\sum_{j\le i}r_j\le4r_i$ follows from $r_j\le r_{j+2}/2$.
On the other hand, the determinant formula and
\eqref{eq:IsotropicProjectionVolume} together with $d > n/2$ imply
\begin{equation} \label{eq:MuVsRW2}
R_W(\Lambda,K)
=\frac{(\prod_{j=i}^k r_j^{d_j})^{1/d}}
{\Vol_d(\Pi_WK)^{1/d}}
\ge \frac{r_i}
{\Vol_d(\Pi_WK)^{1/d}}
\ge \frac{r_i}{C}. 
\end{equation}
Combining \eqref{eq:MuVsRW1} and \eqref{eq:MuVsRW2} then gives the claim.
\end{proof}
Next we prove the $O(n\sqrt{\log n})$ bound on the flatness constant.  
\begin{proof}[Proof of \cref{cor:SubspaceFlatness,cor:Flatness}]
Take $a,E$ from \cref{lem:HalfDimensionalCovering} and set
$P=\Pi_{E^\perp}$. Translation invariance and the triangle inequality
for covering radii~\cite[Lemma~40]{ReisRothvoss2026} give
\[
\mu(\Lambda,K)
\le\mu(\Lambda\cap E,(K-a)\cap E)+\mu(P\Lambda,PK).
\]
Iterate this inequality on the quotient $(P\Lambda,PK)$, whose dimension
is at most $n/2$. If a one-dimensional quotient is reached, its
covering radius equals the determinant of its lattice divided by the
length of its interval, so it contributes the corresponding ratio $R_W$.
We have
\[
R_W(P\Lambda,PK)=R_W(\Lambda,K),\quad\text{for }W\subseteq E^\perp,
\]
since $\Pi_WP=\Pi_W$. Thus, as in
\cite[Corollary~56]{ReisRothvoss2026}, the recursion produces
nonzero subspaces $W_0,\ldots,W_{m-1}$ with
$\dim W_j\le n/2^j$ and $m\le 1+\lfloor\log_2n\rfloor$ such that
\[
\mu(\Lambda,K)
\le C\sqrt{\log n}
\sum_{j=0}^{m-1}R_{W_j}(\Lambda,K).
\]
Bounding every summand by $\sup_W R_W(\Lambda,K)$ proves
\cref{cor:SubspaceFlatness}.

Recall the lattice width $w_\Lambda(K)$ defined in Section~\ref{sec:1}.
By \cite[Lemma~66]{ReisRothvoss2026},
$w_\Lambda(K)\cdot R_{W_j}(\Lambda,K)\le C\dim W_j$. Hence
\[
\mu(\Lambda,K)\cdot w_\Lambda(K)
\le C\sqrt{\log n}\sum_{j=0}^{m-1}\dim W_j
\le Cn\sqrt{\log n}.
\]
Suppose that $\intr(K)\cap\Lambda=\varnothing$.
We claim that $\mu(\Lambda,K)\ge1$.
Otherwise, choose $\mu(\Lambda,K)<r<1$ and $a\in\intr(K)$.
Since $\Lambda+rK=\RR^n$, there are $\lambda\in\Lambda$ and $z\in K$
such that $-(1-r)a=\lambda+rz$. But then
\[
    -\lambda=(1-r)a+rz\in(1-r)a+rK\subseteq\intr(K),
\]
which is a contradiction. Thus $\mu(\Lambda,K)\ge1$, and the preceding
estimate gives
\[
    w_\Lambda(K)\le Cn\sqrt{\log n}.
\]
This proves \cref{cor:Flatness}.
\end{proof}

\subsection{Subgaussian rounding}

Next, we prove \cref{thm:SubgaussianDistrForStableLattice}.
We recall a result by Dadush, Garg, Lovett and Nikolov~\cite[Theorem~3.4]{DadushGargLovettNikolov2019}
which says that for any finite set $S \subseteq \RR^n$ the following conditions are equivalent:
\begin{enumerate}
\item[(i)] There is a random vector $X$ supported on $S$ and an absolute constant $C>0$ such that
$\PP(\abs{\inner{u}{X}}\ge t)\le2e^{-t^2/C}$ for every $\norm{u}_2=1$ and $t\ge0$.
\item[(ii)] There is an absolute constant $C'>0$ such that
$S\cap C'K\neq\varnothing$ for every symmetric convex body $K\subseteq\RR^n$
with $\gamma_n(K)\ge1/2$.
\end{enumerate}
Here $\gamma_n$ is the standard Gaussian measure. The following coupling gives both
the desired expectation and the subgaussian bound.
\begin{theorem}
Let $\Lambda\subseteq\RR^n$ be a lattice such that $\Lambda^*$ is weakly stable
and let $x\in\RR^n$. There is a random vector $Y\in\Lambda$ such that
$\E\, Y=x$ and $Y-x$ is $\sqrt{2e/\pi}$-subgaussian.
\end{theorem}
\begin{proof}
Set $\alpha=2\pi/e$ and let $\mu$ have the even, $\alpha$-strongly
log-concave density $f$ from \cref{thm:1.1}. For $y\in\RR^n$, define
the probability measure $\mu_y$ on $y+\Lambda$ by
\[
    \mu_y(\{y+\lambda\})=\det(\Lambda)f(y+\lambda),
    \quad\text{for }\lambda\in\Lambda.
\]
The masses sum to one by \cref{thm:1.1}, and $\mu_y$ depends only on
the coset $y+\Lambda$. Draw $X\sim\mu$ and, conditionally on $X$, draw $Z\sim\mu_{x+X}$. Set
$Y=x+X-Z$. Since $Z\in x+X+\Lambda$ almost surely, we have
$Y\in\Lambda$. We claim that $Z\sim\mu$. Indeed, $x+X+\Lambda$ is
uniform, so for a fundamental domain $F$ and every bounded measurable
function $h$,
\[
    \E\, h(Z)
    =\frac{1}{\det\Lambda}\int_F\sum_{\lambda\in\Lambda}
      \det(\Lambda)f(y+\lambda)h(y+\lambda)\,dy
    =\int_{\RR^n}f(y)h(y)\,dy.
\]
It follows that $Z\sim\mu$, and hence
\[
    \E\,Y=x+\E\,X-\E\,Z=x.
\]
We now prove the subgaussian bound. For every $u\in\RR^n$,
Cauchy--Schwarz and \cref{lem:ConcentrationForAlphaLogConcave}(iii) give
\begin{align*}
    \E\, e^{\inner{u}{Y-x}}
    &=\E\,\bigl(e^{\inner{u}{X}}\cdot e^{-\inner{u}{Z}}\bigr)\\
    &\le\bigl(\E\, e^{\inner{2u}{X}}\cdot\E\, e^{-\inner{2u}{Z}}\bigr)^{1/2}
    \le e^{2\norm{u}_2^2/\alpha}.
\end{align*}
Hence $Y-x$ is $2/\sqrt{\alpha}=\sqrt{2e/\pi}$-subgaussian.
\end{proof}

\subsection{Reverse Minkowski}

We first derive the correct order by counting short lattice vectors,
then prove
\cref{thm:ReverseMinkowskiWithTightConstants} by spectral comparison.

\begin{lemma}\label{lem:NumOfShortVectorsInWeaklyStableLattice}
There is an absolute constant $C>0$ such that for every weakly stable
lattice $\Lambda\subseteq\RR^n$ and every $r\ge1$,
\[
    \abs{\Lambda\cap rB_2^n}\le n^{Cr^2}.
\]
\end{lemma}
\begin{proof}
Since $(\Lambda^*)^*=\Lambda$ is weakly stable, \cref{thm:1.1}
applies to $\Lambda^*$. It gives a centered random vector $Y$ with an
$\alpha$-strongly log-concave density, where $\alpha=2\pi/e$, such that
$Y+\Lambda^*$ is uniform. By
\cref{lem:ConcentrationForAlphaLogConcave}(iv), for every integer
$p\ge1$ and $x\in\RR^n$,
\[
    \E\,\abs{\inner{Y}{x}}^{2p}\le2p!\,\norm{x}_2^{2p},
\]
since $\alpha\ge2$. Let $\mu$ be the law of $Y$. On the complex Hilbert
space $L^2(\mu)$ of square-integrable functions, identified up to
$\mu$-almost-everywhere equality, write
\[
    \inner{g}{h}_\mu=\E\,\bigl(g(Y)\overline{h(Y)}\bigr),
    \qquad \norm{g}_\mu=\inner{g}{g}_\mu^{1/2}.
\]
For $\lambda\in\Lambda$, define the character
$\chi_\lambda(y)=e^{2\pi i\inner{\lambda}{y}}$.

We first claim that for $\lambda_1,\lambda_2\in\Lambda$,
$\inner{\chi_{\lambda_1}}{\chi_{\lambda_2}}_\mu=\bm{1}_{\lambda_1=\lambda_2}$.
Indeed, let $B$ be a basis matrix for $\Lambda$ and write $\lambda_1-\lambda_2=Ba$, where
$a\in\ZZ^n$. Since $Y+\Lambda^*$ is uniform and
$F=(B^{-1})^{\mathsf T}[0,1)^n$ is a fundamental domain of $\Lambda^*$,
\[
    \inner{\chi_{\lambda_1}}{\chi_{\lambda_2}}_\mu
    =\frac1{\det\Lambda^*}\int_F e^{2\pi i\inner{\lambda_1-\lambda_2}{y}}\,dy
    =\prod_{j=1}^n\int_0^1e^{2\pi ia_jt}\,dt
    =\bm{1}_{a=0}=\bm{1}_{\lambda_1=\lambda_2}.
\]

Set $d=\lceil C'r^2\rceil$, where $C'>0$ is a sufficiently large
absolute constant, and let
\[
    V=\{P\in L^2(\mu):P\text{ is a polynomial of degree at most }d\}.
\]
Then $\dim V\le\binom{n+d}{d}\le(n+1)^d$. Let $\Pi_V$ be the
orthogonal projection onto $V$, and set $S=\Lambda\cap rB_2^n$.

We also claim that for every $\lambda\in S$,
$\norm{\Pi_V\chi_\lambda}_\mu^2\ge1/2$.
Indeed, the degree-$d$ Taylor polynomial of $e^{2\pi ix}$ satisfies
\begin{equation}\label{eq:ApproxPolynomial}
    P(x)=\sum_{k=0}^d\frac{(2\pi ix)^k}{k!},
    \qquad
    \abs{e^{2\pi ix}-P(x)}\le\frac{\abs{2\pi x}^{d+1}}{(d+1)!},
    \quad\text{for }x\in\RR.
\end{equation}
For $\lambda\in S$, put $Q_\lambda(y)=P(\inner{\lambda}{y})\in V$. The moment bound gives
\begin{align*}
    \norm{\chi_\lambda-Q_\lambda}_\mu^2
    &\le\frac{(2\pi)^{2(d+1)}\cdot\E\,\abs{\inner{\lambda}{Y}}^{2(d+1)}}
      {((d+1)!)^2}\\
    &\le\frac{2(2\pi r)^{2(d+1)}}{(d+1)!}
    \le2\left(\frac{4\pi^2e r^2}{d+1}\right)^{d+1}
    \le\frac12,
\end{align*}
for sufficiently large $C'$. Since an orthogonal projection gives the
best approximation in $V$,
\[
    \norm{\Pi_V\chi_\lambda}_\mu^2
    =1-\norm{\chi_\lambda-\Pi_V\chi_\lambda}_\mu^2
    \ge1-\norm{\chi_\lambda-Q_\lambda}_\mu^2\ge\frac12.
\]
By these two claims and Bessel's inequality,
\[
    \abs{S}\le2\sum_{\lambda\in S}\norm{\Pi_V\chi_\lambda}_\mu^2
    \le2\dim V\le2(n+1)^d\le n^{Cr^2}
\]
for a suitable absolute constant $C$.
\end{proof}

This dimension argument above is related to the proof of reverse Minkowski
for integral lattices by Regev and
Stephens-Davidowitz~\cite{RegevStephensDavidowitz2026Integral}.
A lattice is \emph{integral} if $\inner{\lambda_1}{\lambda_2}\in\ZZ$ for all
$\lambda_1,\lambda_2\in\Lambda$. For such a lattice, the points in each shell
$\{\lambda\in\Lambda:\norm{\lambda}_2^2=k\}$ have integral pairwise inner products
in $[-k,k]$. After normalizing each nonzero shell to the unit sphere,
the bound of Delsarte, Goethals and
Seidel~\cite{DelsarteGoethalsSeidel1975}, as presented by
Koornwinder~\cite{Koornwinder1976}, and summation over $k\le r^2$ give
$\abs{\Lambda\cap rB_2^n}\le n^{O(r^2)}$ for $r\ge1$.
Our argument instead uses the inner product induced by $Y$ and
polynomial approximations to the characters defined in the proof.

We next deduce the reverse Minkowski bound from the counting estimate.
Suppose that $\Lambda$ is weakly stable, and set
$s=1/(C_0\sqrt{\log n})$, where $C_0$ will be chosen below.
Weak stability gives
$\norm{\lambda}_2=\det(\ZZ\lambda)\ge1$ for every $0\neq\lambda\in\Lambda$.
Thus \cref{lem:NumOfShortVectorsInWeaklyStableLattice} yields
\begin{align*}
    \rho_s(\Lambda)
    &\le1+\sum_{i=0}^\infty
      \sum_{\substack{\lambda\in\Lambda\\2^i\le\norm{\lambda}_2<2^{i+1}}}
        e^{-\pi\norm{\lambda}_2^2/s^2}\\
    &\le1+\sum_{i=0}^\infty
      n^{(4C-\pi C_0^2)4^i}\le\frac32.
\end{align*}
For the last inequality, choose $\pi C_0^2\ge4C+2$ and use
$\sum_{i\ge0}n^{-2\cdot4^i}\le\sum_{i\ge0}4^{-(i+1)}=1/3$.

We now use the log-concave lattice lift from \cref{thm:MainResult} with Milman's theorem~\cite{Milman2018}. This gives a comparison of the entire
ordered spectrum and, consequently, an explicit bound on the Gaussian mass.

\begin{samepage}
\begin{corollary}[Comparing the eigenvalues of $\RR^n/\Lambda^*$ to that of a Gaussian]\label{cor:SpectralComparison}
Let $\Lambda\subseteq\RR^n$ be a full-rank weakly stable lattice, and set
$\alpha=2\pi/e$. Let $\nu$ be the uniform probability measure on
$\RR^n/\Lambda^*$, and let
\[
    d\gamma_{n,\alpha}(x)
    =\left(\frac{\alpha}{2\pi}\right)^{n/2}
      e^{-\alpha\norm{x}_2^2/2}\,dx.
\]
Write $L_{\Lambda^*}=-\Delta$ for the nonnegative Laplacian on
$L^2(\nu)$ and
$L_\alpha=-\Delta+\alpha\inner{x}{\nabla}$ for the nonnegative
Ornstein--Uhlenbeck operator on $L^2(\gamma_{n,\alpha})$.
Index eigenvalues in nondecreasing order, with multiplicity, starting at
$\lambda_0=0$. Then, for every $k\ge0$,
\begin{equation}\label{eq:SpectralComparison}
    \lambda_k(L_{\Lambda^*})\ge\lambda_k(L_\alpha).
\end{equation}
Consequently, for every $s>0$,
\begin{equation}\label{eq:GaussianMassComparison}
    \rho_s(\Lambda)
    \le \left(1-e^{-1/(2es^2)}\right)^{-n}.
\end{equation}
\end{corollary}
\end{samepage}
\begin{proof}
Apply \cref{thm:1.1} to $\Lambda^*$, whose dual $\Lambda$ is weakly
stable. Let $f$ be the resulting density, set $\mu=f\,dx$ and
$V=-\log f$, and write $\pi(x)=x+\Lambda^*$. Then
\[
    \nabla^2V\succcurlyeq\alpha\Id
    \qquad\text{and}\qquad
    \pi_\#\mu=\nu.
\]
Let $L_V=-\Delta+\inner{\nabla V}{\nabla}$.
Milman's theorem~\cite[Theorem~1.1]{Milman2018} implies that $L_V$ has
discrete spectrum and that for $k \ge 0$ we have
\[
    \lambda_k(L_V)\ge\lambda_k(L_\alpha),\qquad k\ge0.
\]
Pulling $g$ back to $\RR^n$ preserves Rayleigh quotients and the pullback of the span of the first $k+1$ torus
eigenfunctions is a $(k+1)$-dimensional subspace of $H^1(\mu)$ on which
every Rayleigh quotient is at most $\lambda_k(L_{\Lambda^*})$.
The min--max principle therefore yields
\[
    \lambda_k(L_V)\le\lambda_k(L_{\Lambda^*}),
\]
this almost gives us what we want.

Now, the characters $x+\Lambda^*\mapsto e^{2\pi i\inner{\lambda}{x}}$,
indexed by $\lambda\in\Lambda$, form an orthonormal eigenbasis of
$L^2(\nu)$, with eigenvalues $4\pi^2\norm{\lambda}_2^2$.
The eigenvalues of $L_\alpha$ are $\alpha\abs{m}$, indexed by
$m\in\ZZ_{\ge0}^n$, where $\abs{m}=m_1+\cdots+m_n$; see
\cite[Section~2.1]{Milman2018}. Applying the decreasing function
$u\mapsto e^{-tu}$ to \eqref{eq:SpectralComparison} and summing gives
\begin{align*}
    \tr\bigl(e^{-tL_{\Lambda^*}}\bigr)
    &=\sum_{\lambda\in\Lambda}e^{-4\pi^2t\norm{\lambda}_2^2}\\
    &\le\sum_{m\in\ZZ_{\ge0}^n}e^{-\alpha t\abs{m}}
    =\prod_{j=1}^n\left(\sum_{\ell=0}^{\infty}e^{-\alpha t\ell}\right)
    =(1-e^{-\alpha t})^{-n},
\end{align*}
and we are done. Substituting
$t=1/(4\pi s^2)$ and $\alpha=2\pi/e$ gives
\eqref{eq:GaussianMassComparison}.
\end{proof}

\begin{proof}[Proof of \cref{thm:ReverseMinkowskiWithTightConstants}]
For $0<s\le 1/\sqrt{2e\log(3n)}$, \cref{cor:SpectralComparison} gives
\[
    \rho_s(\Lambda)
    \le\left(1-e^{-1/(2es^2)}\right)^{-n}
    \le\left(1-\frac{1}{3n}\right)^{-n}
    \le\frac32. \qedhere
    \]
\end{proof}

\bibliographystyle{amsplain0.bst}
\bibliography{main}

\end{document}